\documentclass[a4paper,twoside]{amsart}
\usepackage[hmargin={25mm,25mm},vmargin={25mm,25mm}]{geometry}
\usepackage{amsmath,amssymb,amsfonts,latexsym,enumerate}
\usepackage{amsthm}
\usepackage[T1]{fontenc}
\usepackage[cp1250]{inputenc}
\usepackage{color}
\newcommand{\E}{\mathbb{E}\,}

\renewcommand{\leq}{\leqslant}
\renewcommand{\ge}{\geqslant}
\renewcommand{\geq}{\geqslant}
\newtheorem{theorem}{Theorem}[section]
 \newtheorem{corollary}[theorem]{Corollary}
 \newtheorem{lemma}[theorem]{Lemma}
 \newtheorem{proposition}[theorem]{Proposition}

 \newtheorem{remark}[theorem]{Remark}
 \newtheorem{example}[theorem]{Example}
\newtheorem{definition}[theorem]{Definition}

\newcommand{\N}{\mathbb{N}}

\newcommand{\lcx}{\leq_{\text{\rm cx}}}

\newcommand{\lst}{\leq_{\text{\rm st}}}
\newcommand{\gst}{\geq_{\text{\rm st}}}

\newcommand{\C}{\mathcal{C}}

\numberwithin{equation}{section}

\begin{document}



\title{Convex order for convolution polynomials of~Borel measures}

\author{Andrzej Komisarski}
\address{Andrzej Komisarski, Department of Probability Theory and Statistics, Faculty of Mathematics and Computer Science,
University of \L\'od\'z, ul. Banacha 22, 90-238 \L\'od\'z, Poland}
\email{andkom@math.uni.lodz.pl}
\author{Teresa Rajba}
\address{Teresa Rajba, University of Bielsko-Biala, Department of Mathematics,  ul. Willowa 2, 43-309 Bielsko-Bia\l{}a, Poland}
\email{trajba@ath.bielsko.pl}


\maketitle

\begin{abstract}
We give necessary and sufficient conditions for Borel measures
to satisfy the inequality introduced by Komisarski, Rajba (2018).
This inequality is a generalization of the convex order inequality for binomial distributions, which was proved
by Mrowiec, Rajba, W\k{a}sowicz (2017), as a probabilistic version of the inequality for convex functions,
that was conjectured as an old open problem by I.~Ra\c{s}a.
We present also further generalizations using convex order inequalities between convolution polynomials of finite Borel measures.
We generalize recent results obtained by B.~Gavrea (2018) in the discrete case to general case. We give solutions to his open problems
and also formulate new problems.

\end{abstract}

\section{Introduction}
Let $\mu$ and $\nu$ be two finite Borel measures (e.g. probability distributions) on $\mathbb R$
with finite first moments (i.e. $\int|x|\ \mu(dx)<\infty$ and the same for $\nu$). We say that 
$\mu$ is \emph{smaller than $\nu$ in the convex order} (denoted as $\mu\lcx\nu$) if $$\int_{\mathbb R}\varphi(x)\mu(dx)\leq\int_{\mathbb R}\varphi(x)\nu(dx) \quad \text{for all convex functions }\ \varphi\colon\mathbb R\to\mathbb R$$
Note that both integrals always exist (finite or infinite). 

Let $P$ and $Q$ be two real polynomials of $m$ variables.
They can be treated as convolution polynomials of finite Borel measures $\mu_1,\dots,\mu_m$ (product of variables corresponds to convolution of measures).
We are interested, when the relation $P(\mu_1,\dots,\mu_m)\lcx Q(\mu_1,\dots,\mu_m)$ holds.

Our investigation is motivated by the recent result of J.\ Mrowiec, T.\ Rajba and S.\ W\k{a}sowicz \cite{MRW2017}
who proved the following convex ordering relation for convolutions of binomial distributions $B(n,x)$ and $B(n,y)$ ($n\in\N$, $x,y\in[0,1]$):
\begin{equation}\label{eq:mainv3}
B(n,x)*B(n,y)\lcx\tfrac12(B(n,x)*B(n,x)+B(n,y)*B(n,y)),
\end{equation}
which is a probabilistic version of the inequality involving Bernstein polynomials and convex functions, that was conjectured as an open problem by 
I.\ Ra\c{s}a \cite{Rasa2014b} (see also \cite{Abel2016},  \cite{AbelRasa2017}, \cite{KomRaj2018bis},  \cite{Gav2018} for further results on the I.\ Ra\c{s}a  problem).

In \cite{KomRaj2018}, we gave a generalization of the inequality \eqref{eq:mainv3}. We introduced and studied  the following convex ordering relation
\begin{equation}\label{eq:mainv2}
\mu*\nu\lcx\tfrac12(\mu*\mu+\nu*\nu),
\end{equation}
where $\mu$ and $\nu$ are two probability distributions on $\mathbb R$. The inequality \eqref{eq:mainv2} can be regarded as the Ra\c{s}a type inequality.
In \cite{KomRaj2018}, we proved Theorem 2.3 providing a~very useful sufficient condition for verification that $\mu$ and $\nu$ satisfy \eqref{eq:mainv2}.
We applied Theorem 2.3 for $\mu$ and $\nu$ from various families of probability distributions.
In particular, we obtained a new proof for binomial distributions, which is significantly simpler and shorter than that given in \cite{MRW2017}. 
By \eqref{eq:mainv2}, we can also obtain inequalities related to some approximation operators associated with $\mu$ and $\nu$.
(such as Bernstein-Schnabl operators, Mirakyan-Sz\'asz operators, Baskakov operators and others, cf.~\cite{KomRaj2018}).

In \cite{KomRaj2018}, we considered also a generalization of \eqref{eq:mainv2}, taking a finite sequence of probability distributions in place of two probability distributions $\mu$ and $\nu$. We proved the Muirhead type inequality for convex orders for convolution polynomials, and we gave a strong generalization of Theorem 2.3.

If $\mu$ and $\nu$ are discrete probability distributions concentrated on the set of non-negative integers $\{0,1,2, \ldots  \}$ with $a_k=\mu (\{ k\})$ and $b_k=\nu (\{ k\})$ ($k=0,1,2, \ldots  $), then the inequality \eqref{eq:mainv2} is equivalent to the following inequality

\begin{equation}\label{eq:Rasatres}
 \sum_{i=0}^n \sum_{j=0}^n \left(a_i \:a_j+b_i \:b_j\right)\varphi\left(i+j \right)\geq
\sum_{i=0}^n \sum_{j=0}^n 2 \:a_i \:b_j \:\varphi\left(i+j \right)
\end{equation}
\smallskip

\noindent for all convex functions $\varphi: \mathbb R \to \mathbb R$.

B.~Gavrea \cite{Gav2018} studied the inequality \eqref{eq:Rasatres} with a convex function $\varphi:\mathbb R \to \mathbb R$
and non-negative sequences $(a_k)$, $(b_k)$ such that $\sum_{k}a_k=\sum_{k}b_k=1$ and $\sum_{k}a_kt^k<\infty$, $\sum_{k}b_kt^k<\infty$ for some $t>1$.
He gave necessary and sufficient conditions for the sequences $(a_k)$, $(b_k)$ to satisfy \eqref{eq:Rasatres}. 
He did not use probabilistic methods. Instead, he used complex analysis.

In Section 2, we give necessary and sufficient conditions for \eqref{eq:mainv2}. We do not limit ourselves to the discrete case.
In our considerations, $\mu$ and $\nu$ are finite Borel measures on $\mathbb R$.
In the particular case of discrete probability distributions,
our assumptions on the sequences $(a_k)$ and $(b_k)$
are weaker then those given in \cite{Gav2018}.

In Section 3, we consider a generalization of \eqref{eq:mainv2} for more than two measures.
As a generalization of results from \cite{KomRaj2018}, we present the Ra\c{s}a type inequalities
for convex orders for convolution polynomials of finite Borel measures on $\mathbb R$.

In Section 4, we give solutions to B.~Gavrea's problems (presented in \cite{Gav2018}) and list some new problems.

\section{The basic case of two measures}

In the sequel we adapt some notation from theory of probability and stochastic orders (see \cite{Shaked2007}).
Let $\mu$ be a finite Borel measure (e.g. a~probability distribution) on $\mathbb R$.
For $x\in\mathbb R$ the delta symbol $\delta_x$ denotes the one-point probability distribution satisfying $\delta_x(\{x\})=1$.
Function $F(x)=F_\mu(x)=\mu((-\infty,x])$ ($x\in\mathbb R$) is the cumulative distribution function of $\mu$ (simply the distribution function).
The complementary cumulative distribution function, or simply the tail distribution, is defined as $\overline F(x)=\mu(\mathbb R)-F(x)=\mu((x,\infty))$.
If $\mu$ and $\nu$ are finite Borel measures such that $\mu(\mathbb R)=\nu(\mathbb R)$
and $F_\mu(x)\geq F_\nu(x)$ for all $x\in\mathbb R$,
then $\mu$ is said to be \emph{smaller than $\nu$ in the~usual stochastic order} (denoted by $\mu\lst\nu$).
An important characterization of the~usual stochastic order for probability distributions is given in the following theorem.
\begin{theorem}[\cite{Shaked2007}, p. 5]\label{th:1a1}
Two probability distributions $\mu$ and $\nu$ satisfy $\mu\lst\nu$ if, and only if, there exist two random variables $X$ and $Y$
defined on the same probability space, such that the distribution of $X$ is $\mu$, the distribution of $Y$ is $\nu$
and $P(X\leq Y)=1$.
\end{theorem}


In \cite{KomRaj2018}, we gave a very useful sufficient condition, that can be used for the verification of the inequality \eqref{eq:mainv2}.
\begin{theorem}[\cite{KomRaj2018}]\label{th:condition}
Let $\mu$ and $\nu$ be two probability distributions with finite first moments, such that $\mu\lst\nu$ or $\nu\lst\mu$. Then
\begin{equation}\label{eq:main}
\mu*\nu\lcx\tfrac12(\mu*\mu+\nu*\nu).
\end{equation}
\end{theorem}

As an application of Theorem~\ref{th:condition}, we obtain that \eqref{eq:main} holds for $\mu$ and $\nu$ from various families of probability distributions: binomial, Poisson, negative binomial, beta, gamma and Gaussian distributions.

The condition presented in Theorem~\ref{th:condition} is sufficient but it is not necessary.
In the following theorem we give a necessary and sufficient condition for finite Borel measures $\mu$ and $\nu$ to satisfy the inequality \eqref{eq:main}.
\begin{theorem}\label{th:necsuf}
Let $\mu$ and $\nu$ be two finite Borel measures on $\mathbb R$ with finite first moments. Let $F$ and $G$
be the distribution functions corresponding to $\mu$ and $\nu$, respectively. Then the following conditions are equivalent:
\begin{itemize}
\item[(1)] $\mu(\mathbb R)=\nu(\mathbb R)$ and $(F-G)*(F-G)\geq0$,
\item[(2)] $\mu*\nu\lcx\frac12(\mu*\mu+\nu*\nu)$.
\end{itemize}
\end{theorem}
\begin{proof}
First we show that (2) implies $\mu(\mathbb R)=\nu(\mathbb R)$.
For the convex function $\varphi(x)=1$ ($x\in \mathbb R$) we have:
$$2\mu(\mathbb R)\nu(\mathbb R)=\int_{-\infty}^\infty 1\ (2\mu*\nu)(dx)\leq\int_{-\infty}^\infty 1\ (\mu*\mu+\nu*\nu)(dx)=(\mu(\mathbb R))^2+(\nu(\mathbb R))^2.$$
In turn, taking the convex function $\varphi(x)=-1$ ($x\in \mathbb R$) we obtain:
$$-2\mu(\mathbb R)\nu(\mathbb R)=\int_{-\infty}^\infty (-1)\ (2\mu*\nu)(dx)\leq\int_{-\infty}^\infty(-1)\ (\mu*\mu+\nu*\nu)(dx)=-(\mu(\mathbb R))^2-(\nu(\mathbb R))^2.$$
Consequently, we have $2\mu(\mathbb R)\nu(\mathbb R)=(\mu(\mathbb R))^2+(\nu(\mathbb R))^2$, which implies $(\mu(\mathbb R)-\nu(\mathbb R))^2=0$.
It follows that $\mu(\mathbb R)=\nu(\mathbb R)$. It remains to show that if $\mu(\mathbb R)=\nu(\mathbb R)$,
then (2) is equivalent to $(F-G)*(F-G)\geq0$.

The relation (2) is equivalent to fulfilling the following inequality
\begin{equation}\label{rasa_v4}
\int_{-\infty}^\infty \varphi(x)(\mu*\mu+\nu*\nu)(dx)\geq\int_{-\infty}^\infty \varphi(x)(2\mu*\nu)(dx)
\end{equation}
for all convex functions $\varphi:\mathbb R \to\mathbb R$.
Note that every convex function $\varphi$ is a pointwise limit
of an~increasing sequence $(\varphi_n)$ of convex, piecewise linear functions.
Therefore (due to monotone convergence theorem for integrals)
\eqref{rasa_v4} is valid for convex functions if, and only if, it is valid
for convex, piecewise linear functions.
On the other hand, every convex, piecewise linear function is a linear combination with non-negative coefficients
of a linear (affine) function and functions of the form \eqref{rasa_v5} (see below).
If $\varphi(x)=ax+b$ is a~linear function, then we have equality in \eqref{rasa_v4}
(both sides of~\eqref{rasa_v4} are equal to $2b(\mu(\mathbb R))^2+2a\mu(\mathbb R)(\int x\mu(dx)+\int x\nu(dx))$).
It follows that \eqref{rasa_v4} is valid for convex functions if, and only if, it is valid for functions of the form
\begin{equation}\label{rasa_v5}
\varphi(x)=(x-A)_+=\max(x-A,0),
\end{equation}
where $A\in\mathbb R$.

In the following computation symbols $\overline F$ and $\overline G$ stand
for the tail distributions of $\mu$ and $\nu$, respectively.
We use the Fubini Theorem several times.
We assume that all the definite integrals are integrals on open intervals. Let $\lambda$ denote the Lebesgue measure on the real line $\mathbb{R}$.
Besides the positive measures $\mu$ and $\nu$, we also study the signed measure $\mu-\nu$. 
Integrability of all considered functions and applicability of the Fubini Theorem
follows from our assumption that $\mu$ and $\nu$ have finite first moments
(e.g., $\mu(\mathbb R)=\nu(\mathbb R)$ implies $F-G=\overline G-\overline F$,
hence
$\int_{-\infty}^\infty|F(t)-G(t)|\ \lambda(dt)\leq
\int_{-\infty}^0(F(t)+G(t))\ \lambda(dt)
+\int_0^{\infty}(\overline F(t)+\overline G(t))\ \lambda(dt)=
\int_{-\infty}^\infty|x|\ \mu(dx)+\int_{-\infty}^\infty|x|\ \nu(dx)<\infty$).

Let $A\in\mathbb R$. Then we have
\begin{multline*}
\int_{-\infty}^\infty (x-A)_+(\mu*\mu+\nu*\nu-2\mu*\nu)(dx)=\int_A^\infty\int_A^x1\ \lambda(dz)(\mu-\nu)^{*2}(dx)=\bigg[A<z<x\bigg]=\\
\int_A^\infty\int_z^\infty1\ (\mu-\nu)^{*2}(dx)\lambda(dz)=
\int_A^\infty\int_{-\infty}^\infty\int_{z-v}^\infty1\ (\mu-\nu)(du)(\mu-\nu)(dv)\lambda(dz)=\\
\bigg[A<z<u+v\bigg]=\int_{-\infty}^\infty\int_{A-v}^\infty\int_A^{u+v}1\ \lambda(dz)(\mu-\nu)(du)(\mu-\nu)(dv)=\\
\bigg[\text{we substitute }t=z-u\bigg]=
\int_{-\infty}^\infty\int_{A-v}^\infty\int_{A-u}^v1\ \lambda(dt)(\mu-\nu)(du)(\mu-\nu)(dv)=\bigg[A<u+t<u+v\bigg]=\\
\int_{-\infty}^\infty\int_t^\infty\int_{A-t}^\infty1\ (\mu-\nu)(du)(\mu-\nu)(dv)\lambda(dt)=
\int_{-\infty}^\infty(\overline F(t)-\overline G(t))(\overline F(A-t)-\overline G(A-t))\ \lambda(dt)=\\
((\overline F-\overline G)*(\overline F-\overline G))(A)=((F-G)*(F-G))(A).
\end{multline*}
The above identity completes the proof of the theorem.
\end{proof}

\begin{remark}
Theorem \ref{th:necsuf} is a generalization of Theorem \ref{th:condition}. Indeed, if $\mu\lst\nu$ or $\nu\lst\mu$, then obviously the condition (1) in Theorem \ref{th:necsuf} is satisfied.
\end{remark}

In the following proposition we give more precise estimation of the difference of integrals given in \eqref{rasa_v4}.
\begin{proposition}
Let $\mu$ and $\nu$ be two finite Borel measures on $\mathbb R$ with finite first and second moments.
Let $\overline\mu=\int_{\mathbb R}x\ \mu(dx)$ and $\overline\nu=\int_{\mathbb R}x\ \nu(dx)$.
Assume, that $\mu*\nu\lcx\frac12(\mu*\mu+\nu*\nu)$ and $\varphi$ is a~twice differentiable convex function.
If both sides of \eqref{rasa_v4} are finite, then
$$\inf_x\varphi''(x)\cdot(\overline\mu-\overline\nu)^2\leq\int_{\mathbb R}\varphi(x)(\mu*\mu+\nu*\nu-2\mu*\nu)(dx)\leq\sup_x\varphi''(x)\cdot(\overline\mu-\overline\nu)^2$$
(we set $\infty\cdot0=\infty$ and $-\infty\cdot0=-\infty$).
\end{proposition}

\begin{proof}
The proofs of both inequalities are similar, therefore we will prove only the left one.
If $\inf_x\varphi''(x)=-\infty$, then the inequality is obvious (the left side is $-\infty$). Assume that 
$m:=\inf_x\varphi''(x)$ is finite. Then $x\mapsto\varphi(x)-\frac m2\cdot x^2$ is a convex function. We have
$$\int_{\mathbb R}\left(\varphi(x)-\tfrac m2\cdot x^2\right)(\mu*\mu+\nu*\nu-2\mu*\nu)(dx)\geq0,$$
which implies
$$\int_{\mathbb R}\varphi(x)(\mu*\mu+\nu*\nu-2\mu*\nu)(dx)\geq
\tfrac m2\cdot\int_{\mathbb R}x^2(\mu*\mu+\nu*\nu-2\mu*\nu)(dx)=m\cdot(\overline\mu-\overline\nu)^2.$$
We need to justify the last equality. If $\mu=\nu=0$, then there is nothing to do.
If $a:=\mu(\mathbb R)=\nu(\mathbb R)>0$, then $\frac\mu a$ and $\frac\nu a$ are probability distributions.
We consider independent random variables $X_1,X_2,Y_1,Y_2$ such that the distribution
of $X_1$ and $X_2$ is $\frac\mu a$ and the distribution of $Y_1$ and $Y_2$ is $\frac\nu a$.
Then we have
\begin{multline*}
\int_{\mathbb R}x^2(\mu*\mu+\nu*\nu-2\mu*\nu)(dx)=
a^2\cdot\int_{\mathbb R}x^2\left(\tfrac\mu a*\tfrac\mu a+\tfrac\nu a*\tfrac\nu a-2\tfrac\mu a*\tfrac\nu a\right)(dx)=\\
a^2\left(\mathbb E(X_1+X_2)^2+\mathbb E(Y_1+Y_2)^2-2\mathbb E(X_1+Y_1)^2\right)=\\
a^2\left(\mathbb EX_1^2+\mathbb EX_2^2+2\mathbb EX_1\mathbb EX_2+\mathbb EY_1^2+\mathbb EY_2^2+2\mathbb EY_1\mathbb EY_2-2\mathbb EX_1^2-2\mathbb EY_1^2-4\mathbb EX_1\mathbb EY_1\right)=\\
2a^2(\mathbb EX_1-\mathbb EY_1)^2=2(\overline\mu-\overline\nu)^2.
\end{multline*}

In the proof of the right inequality we use the convexity of the function $x\mapsto\frac M2\cdot x^2-\varphi(x)$,
where $M=\sup_x\varphi''(x)$.
\end{proof}

Consider now the discrete probability distribution $\mu$ concentrated on the set of non-negative integers $\{0,1,2,\ldots\}$, with $a_k=\mu (\{ k\})$ ($k=0,1,2, \ldots  $). Then the probability generating function corresponding to $\mu$ is given by the formula
$$f(z)=\sum_{k=0}^\infty a_kz^k.$$
\begin{theorem}\label{th:discr}
Let $\mu$ and $\nu$ be discrete probability distributions concentrated on the set of non-negative integers $\{0,1,2,\ldots\}$, with $a_k=\mu (\{ k\})$ and $b_k=\nu (\{ k\})$ ($k=0,1,2, \ldots  $).
Assume that $\mu$ and $\nu$ have finite first moments. Let $F$, $f$ and $G$, $g$ be the distribution functions and the generating functions corresponding to $\mu$ and $\nu$, respectively. Then the following conditions are equivalent:

\begin{itemize}
\item[(1)] For all convex functions $\varphi: \mathbb R \to \mathbb R$
\begin{equation}\label{eq:Rasatres_v4}
 \sum_{i=0}^\infty\sum_{j=0}^\infty\left(a_i \:a_j+b_i \:b_j\right)\varphi(i+j)\geq
 \sum_{i=0}^\infty\sum_{j=0}^\infty2a_i \:b_j \:\varphi(i+j).
\end{equation}
\item[(2)] $(F-G)*(F-G)\geq0$.
\item[(3)] $\frac{d^k}{dz^k}\left(\frac{f(z)-g(z)}{z-1}\right)^2\bigg|_{z=0}\geq0$ \quad \text{for}\  $k=0,1,\dots$.
\end{itemize}
\end{theorem}
\begin{proof}
Since \eqref{eq:Rasatres_v4} is equivalent to the relation $\mu*\nu\lcx\frac12(\mu*\mu+\nu*\nu)$,
the equivalence of (1) and (2) clearly follows from Theorem \ref{th:necsuf}. It suffices to prove the equivalence of (2) and (3).

In the following calculations, we use the existence and finiteness of the first moments of the probability distributions $\mu$ and $\nu$,
which implies that all the following series are absolutely convergent for $z\in[-1,1]$ and we can change the summation order.
By the equality $\sum_{k=0}^\infty a_k=1$, we have
$$\frac{f(z)-1}{z-1}=\sum_{k=0}^\infty a_k\frac{z^k-1}{z-1}=\sum_{k=1}^\infty\sum_{i=0}^{k-1}a_kz^i=
\sum_{i=0}^\infty\sum_{k=i+1}^\infty a_kz^i=\sum_{i=0}^\infty\overline F(i)z^i$$
for every $z\in[-1,1)$, where $\overline F$ is the tail distribution of $\mu$.
Note that $\sum_{i=0}^\infty\overline F(i)=\int_{\mathbb R}x\ \mu(dx)<\infty$.
Similarly, we obtain
$$\frac{g(z)-1}{z-1}=\sum_{i=0}^\infty\overline G(i)z^i,$$
where $\overline G$ is the tail distribution of $\nu$. Therefore for every $z\in[-1,1)$ we have
$$\left(\frac{f(z)-g(z)}{z-1}\right)^2=\left(\sum_{i=0}^\infty(\overline F-\overline G)(i)z^i\right)^2
=\left(\sum_{i=0}^\infty(G-F)(i)z^i\right)^2
=\sum_{i=0}^\infty((G-F)\hat{\ast}(G-F))(i)z^i.$$
Here $\hat{\ast}$ denotes the discrete convolution (the Euler product) of sequences.
Condition (3) is equivalent to the non-negativity of all the coefficients in the above series.
Note that if $i\in\mathbb Z$ and $x\in[i,i+1]$, then $((G-F)*(G-F))(x)=(x-i)((G-F)\hat{\ast}(G-F))(i)+(i+1-x)((G-F)\hat{\ast}(G-F))(i-1)$
(here we put $((G-F)\hat{\ast}(G-F))(i)=0$ for $i<0$). Therefore the non-negativity of all the terms $((G-F)\hat{\ast}(G-F))(i)$ is equivalent to (2).
The theorem is proved.
\end{proof}
\begin{remark}
B.~Gavrea \cite{Gav2018} also gave the condition (3) as necessary and sufficient to satisfy the condition (1), but assuming that the radii
of convergence of functions $f$ and $g$ are greater then 1 (in particular, there exist all the moments of $\mu$ and $\nu$).
The assumption in Theorem \ref{th:discr} is weaker, we assume only the existence of the first moments of $\mu$ and $\nu$.
Furthermore, in Theorem \ref{th:necsuf}, we give necessary and sufficient condition for all distributions, not just for discrete.
\end{remark}
\section{The case of $m$ measures}
In this section we consider $m$ finite Borel measures on $\mathbb R$, $\mu_1,\dots,\mu_m$, with finite first moments.
We denote  $F_i(x)=\mu_i((-\infty,x])$ and $\overline F_i(x)=\mu_i((x,\infty))$ for $i=1,\dots,m$
and $x\in\mathbb R$.

Let $P$ and $Q$ be two polynomials of $m$ variables.
They can be treated as convolution polynomials of the measures $\mu_1,\dots,\mu_m$ (product of variables corresponds to convolution of measures).
We are interested, when the relation $P(\mu_1,\dots,\mu_m)\lcx Q(\mu_1,\dots,\mu_m)$ holds.
Since $\lcx$ is defined for non-negative measures, we generally assume that
the polynomials have non-negative coefficients, although in the proofs we  consider also differences of such polynomials.

\begin{proposition}
Let $\mu_1,\dots,\mu_m$ be finite Borel measures on $\mathbb R$ with finite first moments.
For $i=1,\dots,m$, we denote $a_i=\mu_i(\mathbb R)$ and $b_i=\int x\ \mu_i(dx)$.
Let $P,Q$ be polynomials of $m$ variables with non-negative coefficients, such that $P(\mu_1,\dots,\mu_m)\lcx Q(\mu_1,\dots,\mu_m)$.
Then $P(a_1,\dots,a_m)=Q(a_1,\dots,a_m)$ and $\frac{\partial P}{\partial(b_1,\dots,b_m)}(a_1,\dots,a_m)
=\frac{\partial Q}{\partial(b_1,\dots,b_m)}(a_1,\dots,a_m)$ (the directional derivatives
along the vector $(b_1,\dots,b_m)$ at the point $(a_1,\dots,a_m)$).
\end{proposition}
\begin{proof}
Let $P=\sum_jp_j\prod_{i=1}^mx_i^{k_{j,i}}$ and $Q=\sum_jq_j\prod_{i=1}^mx_i^{l_{j,i}}$.
Considering the convex functions $\varphi(x)=1$ and $\varphi(x)=-1$ we obtain
$\int1\ P(\mu_1,\dots,\mu_m)=\int1\ Q(\mu_1,\dots,\mu_m)$, which implies $P(a_1,\dots,a_m)=Q(a_1,\dots,a_m)$.
Taking the convex functions $\varphi(x)=x$ and $\varphi(x)=-x$ we get
$\int x\ P(\mu_1,\dots,\mu_m)=\int x\ Q(\mu_1,\dots,\mu_m)$,
which is equivalent to $\sum_jp_j\sum_{i=1}^mk_{j,i}\frac{b_i}{a_i}\prod_{i=1}^ma_i^{k_{j,i}}
=\sum_jq_j\sum_{i=1}^ml_{j,i}\frac{b_i}{a_i}\prod_{i=1}^ma_i^{l_{j,i}}$.
Consequently we obtain $\frac{\partial P}{\partial(b_1,\dots,b_m)}(a_1,\dots,a_m)
=\frac{\partial Q}{\partial(b_1,\dots,b_m)}(a_1,\dots,a_m)$.
\end{proof}
\begin{theorem}\label{th:muir5}
Let $\mu_1,\dots,\mu_m$ be finite Borel measures on $\mathbb R$ with finite first moments.
We assume that $\mu_1(\mathbb R)=\dots=\mu_m(\mathbb R)$ and $(F_i-F_j)*(F_i-F_j)\geq0$ for all $i,j=1,\dots,m$. 
Let $P$ and $Q$ be polynomials of variables $x_1,\dots,x_m$ with non-negative coefficients, such that $(Q-P)(x_1,\dots,x_m)=\sum_{i\neq j}R_{i,j}(x_1,\dots,x_m)(x_i-x_j)^2$,
where $R_{i,j}$ are polynomials of variables $x_1,\dots,x_m$ with non-negative coefficients.
Then $P(\mu_1,\dots,\mu_m)\lcx Q(\mu_1,\dots,\mu_m)$.
\end{theorem}
\begin{proof}
Since $R_{i,j}$ are polynomials with non-negative coefficients,
it follows that $R_{i,j}(\mu_1,\dots,\mu_m)$ are finite non-negative measures.
Let $\varphi$ be a function which is affine or of form $\varphi(x)=(x-A)_+$, where $A\in\mathbb R$.
Then $\varphi$ is integrable with respect to every polynomial of measures $\mu_1,\dots,\mu_m$ and we have
\begin{multline*}
\int\varphi(x)\ (Q-P)(\mu_1,\dots,\mu_m)(dx)=
\sum_{i\neq j}\int\varphi(x)\ (R_{i,j}(\mu_1,\dots,\mu_m)*(\mu_i-\mu_j)^{*2})(dx)=\\
\sum_{i\neq j}\int\int\varphi(u+v)\ (\mu_i-\mu_j)^{*2}(du)R_{i,j}(\mu_1,\dots,\mu_m)(dv)\geq0,
\end{multline*}
because, by Theorem \ref{th:necsuf}, the internal integral in each component of the above sum is non-negative.

Since other convex functions are limits of convex combinations with non-negative coefficients of functions considered above, we obtain $P(\mu_1,\dots,\mu_m)\lcx Q(\mu_1,\dots,\mu_m)$.
\end{proof}
\begin{remark}
Let $\mu_1,\dots,\mu_k$ be finite Borel measures on $\mathbb R$ with finite first moments.
If $\mu_1,\dots,\mu_k$ are pairwise comparable in the usual stochastic order
(for each $1\leq i,j\leq k$ we have $\mu_i\lst\mu_j$ or $\mu_i\gst\mu_j$),
then the inequalities $(F_i-F_j)*(F_i-F_j)\geq0$ are satisfied for all $i,j=1,\dots,m$.
\end{remark}

Before we state the next theorem, we need to present two definitions.

In the set of all the $m$-tuples $\mathbf p=(p_1,\dots,p_m)$ of non-negative integers we consider the following quasiorder. 
\begin{definition}
We say that $\mathbf q$ majorizes $\mathbf p$ (denoted by $\mathbf p\prec\mathbf q$ or $\mathbf q\succ\mathbf p$) if
\begin{enumerate}
\item $\sum_{l=1}^m\widehat p_l=\sum_{l=1}^m\widehat q_l$,
\item $\sum_{l=1}^k\widehat p_l\leq\sum_{l=1}^k\widehat q_l$ for $k=1,\dots,m$,
\end{enumerate}
where $\mathbf{\widehat p}=(\widehat p_1,\dots,\widehat p_m)$ and $\mathbf{\widehat q}=(\widehat q_1,\dots,\widehat q_m)$
are nonincreasing permutations of $\mathbf p$ and $\mathbf q$, respectively
($\widehat p_1\geq\dots\geq\widehat p_m$ and $\widehat q_1\geq\dots\geq\widehat q_m$).
\end{definition}
The majorization has been studied in \cite{Hardy1952}, \cite{MarshallOlkin2011}, and many other sources.

The following condition ($S$) plays an important role:
We say that a pair $\mathbf p\prec\mathbf q$ satisfies the condition ($S$), if there exist $1\leq l_1<l_2\leq m$
such that $\widehat q_{l_1}=\widehat p_{l_1}+1$, $\widehat q_{l_2}=\widehat p_{l_2}-1$ and $\widehat q_l=\widehat p_l$ for $l\notin\{l_1,l_2\}$.

In \cite{KomRaj2018} we proved the following lemma.
\begin{lemma}\label{lm:muirh}
If $\mathbf p\prec\mathbf q$, then $\mathbf{\widehat p}=\mathbf{\widehat q}$
or there exist $\mathbf p=\mathbf p^0\prec\mathbf p^1\prec\dots\prec\mathbf p^I=\mathbf q$
such that $\mathbf p^{i-1}\prec\mathbf p^i$ satisfies ($S$) for $i=1,\dots,I$.
\end{lemma}

The main theorem of this section concerns polynomials defined as follows.

\begin{definition}
Let $m\in\mathbb N$ and let $\Pi$ be the set of all permutations of the set $\{1,\dots,m\}$.
For every $m$-tuple $\mathbf p=(p_1,\dots,p_m)$ of non-negative integers, we define the following polynomial:
$$W^{\mathbf p}(x_1,\dots,x_m):=\tfrac1{m!}\sum_{\pi\in\Pi}\prod_{l=1}^mx_{\pi(l)}^{p_l}.$$
\end{definition}
Clearly $W^{\mathbf p}$ is a symmetric polynomial with non-negative coefficients.
If $\mathbf q$ is a permutation of $\mathbf p$, then $W^{\mathbf q}=W^{\mathbf p}$.
In particular $W^{\mathbf p}=W^{\mathbf{\widehat p}}$.

\begin{theorem}\label{th:Muir}
Let $m\in\mathbb N$ and let $\mu_1,\dots,\mu_m$ be finite Borel measures on $\mathbb R$
satisfying $\mu_1(\mathbb R)=\dots=\mu_m(\mathbb R)$ and $\int|x|\ \mu_l(dx)<\infty$
for $l=1,\dots,m$.
Let $F_i(x)=\mu_i((-\infty,x])$ for $x\in\mathbb R$.
Assume that $(F_i-F_j)*(F_i-F_j)\geq0$ for all $i$ and $j$.
If $\mathbf p\prec\mathbf q$,
then $W^{\mathbf p}(\mu_1,\dots,\mu_m)\lcx W^{\mathbf q}(\mu_1,\dots,\mu_m)$.
\end{theorem}

\begin{proof}
In view of Lemma \ref{lm:muirh} and transitivity of $\lcx$, it is enough
to consider the case when the pair $\mathbf p\prec\mathbf q$ satisfies condition ($S$).
Let $l_1<l_2$ be the indices given in condition ($S$).
For every $\pi\in\Pi$ we define $\pi'\in\Pi$ by $\pi'(l_1)=\pi(l_2)$, $\pi'(l_2)=\pi(l_1)$
and $\pi'(l)=\pi(l)$ for $l\notin\{l_1,l_2\}$. We have
\begin{multline*}
(W^{\mathbf q}-W^{\mathbf p})(x_1,\dots,x_m)=
\tfrac1{m!}\sum_{\pi\in\Pi}\left(\prod_{l=1}^mx_{\pi(l)}^{\widehat q_l}-\prod_{l=1}^mx_{\pi(l)}^{\widehat p_l}\right)=\\
\tfrac1{m!}\sum_{1\leq u<v\leq m}\sum_{\substack{\pi\in\Pi\\\pi(l_1)=u\\\pi(l_2)=v}}
\left(\prod_{l=1}^mx_{\pi(l)}^{\widehat q_l}+\prod_{l=1}^mx_{\pi'(l)}^{\widehat q_l}
-\prod_{l=1}^mx_{\pi(l)}^{\widehat p_l}-\prod_{l=1}^mx_{\pi'(l)}^{\widehat p_l}\right)=\\
\tfrac1{m!}\sum_{1\leq u<v\leq m}\sum_{\substack{\pi\in\Pi\\\pi(l_1)=u\\\pi(l_2)=v}}\prod_{l\neq l_1,l_2}x_{\pi(l)}^{\widehat p_l}
\left(x_u^{\widehat p_{l_1}+1}x_v^{\widehat q_{l_2}}+x_v^{\widehat p_{l_1}+1}x_u^{\widehat q_{l_2}}
-x_u^{\widehat p_{l_1}}x_v^{\widehat q_{l_2}+1}-x_v^{\widehat p_{l_1}}x_u^{\widehat q_{l_2}+1}\right)
\end{multline*}
Note that
\begin{multline*}
x_u^{\widehat p_{l_1}+1}x_v^{\widehat q_{l_2}}+x_v^{\widehat p_{l_1}+1}x_u^{\widehat q_{l_2}}
-x_u^{\widehat p_{l_1}}x_v^{\widehat q_{l_2}+1}-x_v^{\widehat p_{l_1}}x_u^{\widehat q_{l_2}+1}=\\
(x_u-x_v)\left(x_u^{\widehat p_{l_1}}x_v^{\widehat q_{l_2}}-x_v^{\widehat p_{l_1}}x_u^{\widehat q_{l_2}}\right)=
(x_u-x_v)^2\sum_{j=\widehat q_{l_2}}^{\widehat p_{l_1}-1}x_u^jx_v^{\widehat p_{l_1}+\widehat q_{l_2}-1-j}.
\end{multline*}
It follows that
$$(W^{\mathbf q}-W^{\mathbf p})(x_1,\dots,x_m)=
\sum_{1\leq u<v\leq m}(x_u-x_v)^2\sum_{\substack{\pi\in\Pi\\\pi(l_1)=u\\\pi(l_2)=v}}
\sum_{j=\widehat q_{l_2}}^{\widehat p_{l_1}-1}\frac{x_u^jx_v^{\widehat p_{l_1}+\widehat q_{l_2}-1-j}}{m!}\prod_{l\neq l_1,l_2}x_{\pi(l)}^{\widehat p_l}.
$$
By Theorem \ref{th:muir5}, we obtain $W^{\mathbf p}(\mu_1,\dots,\mu_m)\lcx W^{\mathbf q}(\mu_1,\dots,\mu_m)$.
\end{proof}

\begin{remark}
Theorem \ref{th:Muir} is an analogue of Muirhead's inequality (see \cite{Hardy1952}, Theorem 45 or \cite{MarshallOlkin2011}, Section 3G)
with positive numbers replaced by Borel measures on $\mathbb R$, multiplication replaced by convolution, and $\leq$ replaced by $\lcx$.
Moreover, if $x_1,\dots,x_k>0$, then applying Theorem \ref{th:Muir} with $\mu_l=\delta_{\ln x_l}$ (for $l=1,\dots,k$)
and the convex function $\varphi(x)=e^x$, we obtain the classical Muirhead inequality with integer exponents:
If $\mathbf p\prec\mathbf q$ and $x_1,\dots,x_m>0$, then $W^{\mathbf p}(x_1,\dots,x_m)\leq W^{\mathbf q}(x_1,\dots,x_m)$.
\end{remark}

If we apply Theorem \ref{th:Muir} with $(p)=(1,\dots,1)$ and $(q)=(m,0,\dots,0)$, we get 
the following corollary, which generalizes
Ra\c{s}a type inequalities proved in \cite{MRW2017}, \cite{KomRaj2018} and \cite{Gav2018}.
\begin{corollary}
If $\mu_1,\dots,\mu_m$ are finite Borel measures on $\mathbb R$
satisfying assumptions of Theorem \ref{th:Muir}, then 
$$\mu_1*\dots*\mu_m\lcx\tfrac1m\left[(\mu_1)^{*m}+\dots+(\mu_m)^{*m}\right].$$
In particular
$$\sum_{i_1,\dots,i_m=0}^n\bigl(b_{n,i_1}(x_1)\cdots b_{n,i_m}(x_1)+\dots+b_{n,i_1}(x_m)\dots b_{n,i_m}(x_m)
-mb_{n,i_1}(x_1)\dots b_{n,i_m}(x_m)\bigr)\varphi\left(\tfrac{i_1+\dots+i_m}{mn}\right)\geq0,$$
in the case of $\mu_i=B(n,x_i)$ $(x_i\in[0,1])$ for $i=1,\dots,m.$
\end{corollary}

One might expect that every polynomial inequality valid for non-negative real numbers has its counterpart
for finite Borel measures and convex orders. The following example shows that it is very far from true.
\begin{example}
Let $P(x,y)=\frac12x^3y+\frac12xy^3$ and $Q(x,y)=\frac18x^4+\frac34x^2y^2+\frac18y^4$.
The polynomials $P$ and $Q$ are symmetric and homogeneous polynomials of degree $4$.
We have $Q(x,y)-P(x,y)=\frac18(x-y)^4\geq0$ for every $x,y\in\mathbb R$.
Both $P$ and $Q$ have non-negative coefficients and $P(1,1)=Q(1,1)=1$.
It follows that $P(\mu,\nu)$ and $Q(\mu,\nu)$ are probability distributions
whenever $\mu$ and $\nu$ are probability distributions.
If the expected values (means) $\mathbb E\mu$ and $\mathbb E\nu$ are finite,
then $\mathbb E P(\mu,\nu)=2(\mathbb E\mu+\mathbb E\nu)=\mathbb E Q(\mu,\nu)$.
Despite all this regularity the inequality $P(\mu,\nu)\lcx Q(\mu,\nu)$
is not valid for $\mu=\delta_0$ and $\nu=\frac12\delta_0+\frac12\delta_1$
(note that $F-G\geq 0$, hence $(F-G)*(F-G)\geq 0$, cf.\ assumptions of Theorem 1 and Theorem \ref{th:Muir}).
Indeed, $P(\mu,\nu)=\frac5{16}\delta_0+\frac7{16}\delta_1+\frac3{16}\delta_2+\frac1{16}\delta_3$
and $Q(\mu,\nu)=\frac{41}{128}\delta_0+\frac{52}{128}\delta_1+\frac{30}{128}\delta_2+\frac4{128}\delta_3+\frac1{128}\delta_4$
and for the convex function $\varphi(x)=\max(0,x-2)$ we have $\int\varphi(x)P(\mu,\nu)(dx)=\frac1{16}>\frac6{128}=\int\varphi(x)Q(\mu,\nu)(dx)$,
hence $P(\mu,\nu)\not\lcx Q(\mu,\nu)$.
\end{example}

\section{Open problems}

For $n\in\N$ 
the classical Bernstein operators $B_n:\C([0,1])\to\C([0,1])$, defined by
$$\left(B_n f\right)(x)=\sum_{i=0}^n b_{n,i}(x)f\left(\tfrac in\right)\quad\text{for } x\in[0,1],$$
with the Bernstein basic polynomials 
\[
 b_{n,i}(x)=\binom{n}{i}x^i(1-x)^{n-i} \quad \text{for } \ i=0,1,\dots,n, \  x\in[0,1],
\]
are the most prominent positive linear approximation operators (see \cite{Lorentz1953}).

The inequality \eqref{eq:mainv3} is the probabilistic version of the following
inequality involving Bernstein polynomials and convex functions, that was conjectured as an open problem by 
I.~Ra\c{s}a in~\cite{Rasa2014b}
\begin{equation}\label{MRW_eq:Rasa}
  \sum_{i,j=0}^n\bigl(b_{n,i}(x)b_{n,j}(x)+b_{n,i}(y)b_{n,j}(y)-2b_{n,i}(x)b_{n,j}(y)\bigr)f\left(\frac{i+j}{2n}\right)\ge 0
 \end{equation}
\par\smallskip
\noindent for each convex function $f\in\C\bigl([0,1]\bigr)$ and for all $x,y\in[0,1]$.

Ra\c{s}a \cite{Rasa2017b} remarked, that \eqref{MRW_eq:Rasa} is equivalent to 
\begin{equation}\label{Rasa_eq:Rasabis}
 \left(B_{2n}f\right)(x) + \left(B_{2n}f\right)(y)
\geq 2\;\sum_{i=0}^n \sum_{j=0}^n b_{n,i}(x)b_{n,j}(y)\ f\left(\frac{i+j}{2n}\right).
\end{equation}

B. Gavrea \cite{Gav2018} presented the following generalization of the problem of I.\ Ra\c{s}a \cite{Rasa2014b}.
\par\bigskip
\textbf{Problem 1. \cite{Gav2018}} Let $D=[0,1]\times [0,1]$, $g\in C(D)$ and $n\in\N$. The Bernstein operator is then defined by 
$$\left(B_{n,n}g\right)(x,y)=\sum_{i=0}^n \sum_{j=0}^n b_{n,i}(x)b_{n,j}(y)g\left(\frac{ i}{n},\frac{ j}{n}\right)\quad\text{for } \ (x,y)\in D.$$
Give a characterization of the class of of convex functions $g$ defined on $D$, satisfying
\begin{equation}\label{eq:Gavv1}
\left(B_{n,n}g\right)(x,x)+\left(B_{n,n}g\right)(y,y)-2\left(B_{n,n}g\right)(x,y)\geq 0
\end{equation}
for all $(x,y)\in D$.
\begin{remark}[\cite{Gav2018}]
We note that, if $\varphi\in C([0,1])$ is a convex function, and
$$
g(x,y)=\varphi\left( \frac{x+y}{2} \right) \quad \text{for}  \ (x,y)\in D,
$$
then \eqref{eq:Gavv1} coincides with the Ra\c{s}a inequality \eqref{MRW_eq:Rasa}.
\end{remark}
\begin{remark}
Note that if \eqref{eq:Gavv1} is satisfied for all $(x,y)\in D$, then also
\begin{equation}\label{eq:Gavv10}
\left(B_{n,n}g\right)(x,x)+\left(B_{n,n}g\right)(y,y)-2\left(B_{n,n}g\right)(y,x)\geq 0
\end{equation}
for all $(x,y)\in D$. Adding inequalities \eqref{eq:Gavv1} and \eqref{eq:Gavv10}, we obtain
\begin{equation}\label{eq:Gavv20}
\left(B_{n,n}g\right)(x,x)+\left(B_{n,n}g\right)(y,y)-\left(B_{n,n}g\right)(x,y)-\left(B_{n,n}g\right)(y,x)\geq 0.
\end{equation}
\end{remark}
Taking into account \eqref{eq:Gavv20}, we consider a modification of Problem 1.
\par\bigskip
\textbf{Problem 1'.} Let $D=[0,1]\times [0,1]$, $g\in C(D)$ and $n\in\N$. 
Give a characterization of the class of convex functions $g$ defined on $D$, satisfying
\begin{equation}\label{eq:Gavv2}
\left(B_{n,n}g\right)(x,x)+\left(B_{n,n}g\right)(y,y)-\left(B_{n,n}g\right)(x,y)-\left(B_{n,n}g\right)(y,x)\geq 0 
\end{equation}
for all $(x,y)\in D$.
\par\medskip
\begin{remark}\label{rem:prob}
The inequality \eqref{eq:Gavv2} has the probabilistic interpretation. It is equivalent to the following inequality 
\begin{equation}\label{eq:Gavv3}
\E g\left(\frac{X}{n}, \frac{Y}{n}\right)+\E g\left(\frac{Y}{n}, \frac{X}{n}\right) \leq \E g\left(\frac{X_1}{n}, \frac{X_2}{n}\right)+\E g\left(\frac{Y_1}{n}, \frac{Y_2}{n}\right),
\end{equation}
where $(X,Y)$, $(X_1,X_2)$, $(Y_1,Y_2)$ are pairs of independent random variables such that $X,X_1,X_2\sim B(n,x) $ and $Y,Y_1,Y_2\sim B(n,y)$.
\end{remark}
\par\medskip
We use the following notation: $X\sim\mu$ means that $\mu$ is the probability distribution of a random variable $X$.

The inequality \eqref{eq:Gavv3} is not satisfied for all convex functions $g\in C(D)$. Let us take $g(x,y)=|x-y|$. Then $g$ is convex,
$g(0,0)= g(1,1)=0$ and  $g(0,1)= g(1,0)=1$. Let $X,X_1,X_2\sim B(n,0)=\delta_0$, $Y,Y_1,Y_2\sim B(n,1)=\delta_n$
be independent random variables. We obtain 
$$\E g\left(\frac{X}{n},\frac{Y}{n}\right)+ \E g\left(\frac{Y}{n},\frac{X}{n}\right)=1+1>0+0=\E g\left(\frac{X_1}{n},\frac{X_2}{n}\right)+ \E g\left(\frac{Y_1}{n},\frac{Y_2}{n}\right),$$
consequently the inequality \eqref{eq:Gavv3} is not fulfilled.
\par\bigskip

Let $g: \mathbb R^2 \to  \mathbb R$ be a convex function. We consider the Jensen gap corresponding to $g$ that is given by
\begin{equation}\label{eq:cond2}
\mathcal{J}(g;(x_1, x_2), (y_1, y_2))=\frac{g(x_1,x_2)+g(y_1,y_2)}{2}-g\left( \frac{(x_1,x_2)+(y_1,y_2)}{2} \right)
\end{equation}
\noindent for all $x_1, x_2, y_1, y_2 \in \mathbb R.$
Since $g$ is convex, 
\begin{equation*}\label{eq:cond1}
\mathcal{J}(g;(x_1, x_2), (y_1, y_2))\geq 0\quad \text{and} \quad \mathcal{J}(g;(x_1, y_2), (y_1, x_2))\geq 0
\end{equation*}
for all $x_1, x_2, y_1, y_2 \in \mathbb R$.

We will consider convex functions $g$ satisfying the inequality
\begin{equation}\label{eq:cond}
\mathcal{J}(g;(x_1, x_2), (y_1, y_2))\geq \mathcal{J}(g;(x_1, y_2), (y_1, x_2))
\end{equation}
for all $x_1, x_2, y_1, y_2 \in \mathbb R$ such that 
\begin{equation}(y_1-x_1)(y_2-x_2)>0.\end{equation}
By \eqref{eq:cond2}, we can write \eqref{eq:cond} in the form 
\begin{equation*}
\frac{g(x_1,x_2)+g(y_1,y_2)}{2}-g\left( \frac{(x_1,x_2)+(y_1,y_2)}{2} \right) \geq \frac{g(x_1,y_2)+g(y_1,x_2)}{2}-g\left( \frac{(x_1,y_2)+(y_1,x_2)}{2} \right), 
\end{equation*}
or equivalently 
\begin{equation}\label{eq:cond4}
g(x_1,x_2)+g(y_1,y_2) \geq g(x_1,y_2)+g(y_1,x_2). 
\end{equation}

The following theorem says, that if the convex function $g$ satisfies the inequality \eqref{eq:cond4}, which is equivalent to the inequality \eqref{eq:cond}, and the random variables $X$ and $Y$ (which are not necessary binomially distributed) are chosen to satisfy some sufficient condition, then
the inequality \eqref{eq:Gavv3} is satisfied (up to natural number $n$).
\begin{theorem}\label{th:cond}
Let $X$ and $Y$ be two independent random variables with finite first moments, such that
\begin{equation}\label{eq:cond5}
X\lst Y \quad \text{or} \quad Y\lst X. 
\end{equation}
Let $g: \mathbb R^2 \to  \mathbb R$ be a convex function satisfying
\begin{equation}\label{eq:cond6}
g(x_1,x_2)+g(y_1,y_2) \geq g(x_1,y_2)+g(y_1,x_2)
\end{equation}
for all $x_1, x_2, y_1, y_2 \in \mathbb R$ such that 
\begin{equation}(y_1-x_1)(y_2-x_2)>0.\end{equation}
Then
\begin{equation}\label{eq:cond7}
\E g(X_1,X_2)+\E g(Y_1,Y_2) \geq \E g(X,Y)+\E g(Y,X), 
\end{equation}
where $X_1,X_2$ and $Y_1,Y_2$ are independent random variables such that $X_1,X_2\sim X$ and $Y_1,Y_2\sim Y$.
\end{theorem}
\begin{proof}
Without loss of generality we may assume that $X\lst Y$.
By Theorem \ref{th:1a1}, there exist two independent random vectors $(X_1,Y_1)$ and $(X_2,Y_2)$ such that 
\begin{equation}\label{eq:suf3}
X_1,X_2\sim X, \quad Y_1,Y_2\sim Y, \quad P(X_1\leq Y_1)=1 \quad \text{and} \quad P(X_2\leq Y_2)=1.
\end{equation} 
By \eqref{eq:suf3} and \eqref{eq:cond6} we obtain
$$
P\left( g(X_1,X_2)+g(Y_1,Y_2) \geq  g(X_1,Y_2)+ g(Y_1,X_2) \right) =1
$$
which implies
\begin{equation}\label{eq:suf4}
\E g(X_1,X_2)+\E g(Y_1,Y_2) \geq  \E g(X_1,Y_2)+ \E g(Y_1,X_2).
\end{equation}
Since \eqref{eq:suf4} is equivalent to \eqref{eq:cond7}, the theorem is proved.
\end{proof}

Note that binomially distributed random variables $X\sim B(n,x)$ and $Y\sim B(n,y)$ satisfy the condition \eqref{eq:cond5} (see \cite{KomRaj2018}).
Therefore we obtain (from Theorem \ref{th:cond} and Remark \ref{rem:prob}) the following
sufficient condition for functions $g$ that appear in Problem 1'.
\begin{corollary}
Let $g: [0,1]^2 \to  \mathbb R$ be a convex function satisfying
\begin{equation*}\label{eq:cond8}
g(x_1,x_2)+g(y_1,y_2) \geq g(x_1,y_2)+g(y_1,x_2) 
\end{equation*}
for all $x_1, x_2, y_1, y_2 \in \mathbb R$ such that $(y_1-x_1)(y_2-x_2)>0$.
Then
\begin{equation*}\label{eq:Gavv30}
\left(B_{n,n}g\right)(x,x)+\left(B_{n,n}g\right)(y,y)-\left(B_{n,n}g\right)(x,y)-\left(B_{n,n}g\right)(y,x)\geq 0.
\end{equation*}
\end{corollary}

\par\medskip

We present a new open problem, which is  a generalization of Problem 1 \cite{Gav2018}.
\par\bigskip
\textbf{Problem 3.} Let $k\in\N$, $n_i \in \mathbb{N}$ for $i=1,\ldots k$, $\sum_{i=1}^k n_i=m$, $D=[0,1]^k$ and  $g\in C(D)$. The Bernstein type operator $B_{n_1,\ldots,n_k}$ is defined by
$$\left(B_{n_1,\ldots,n_k}g\right)(x_1,\ldots,x_k)=\sum_{i_1=0}^{n_1}\ldots \sum_{i_k=0}^{n_k}b_{n_1,i_1}(x_1)\ldots b_{n_k,i_k}(x_k) g\left(\frac{ i_1}{n_1},\ldots,\frac{ i_k}{n_k}\right)$$
for  $(x_1,\ldots,x_k)\in D$.

Give a characterization of the class of convex functions $g$ defined on $D$ and satisfying
\begin{equation}\label{eq:Gavv6}
\left(B_{n_1,\ldots,n_k}g\right)(x_1,\ldots,x_k)\leq
\sum_{i=1}^k\frac{n_i}m\;\left(B_{n_1,\ldots,n_k}g\right)\left(x_i,\ldots,x_i\right)
\end{equation}
for all $(x_1,\ldots,x_k)\in D$.

\par\medskip
In \cite{KomRaj2018bis}, we proved the following generalization of the Ra\c{s}a inequality \eqref{Rasa_eq:Rasabis}.
\begin{theorem} [\cite{KomRaj2018bis}, Theorem 2.2, (2.5)]\label{th:2.1}
Let $k\in\N$, $n_i \in \mathbb{N}$ for $i=1,\ldots k$ and $\sum_{i=1}^k n_i=m$. Then
\par\medskip
\begin{equation}\label{eq:6prim}
  \sum_{i_1=0}^{n_1}\ldots \sum_{i_k=0}^{n_k}b_{n_1,i_1}(x_1)\ldots b_{n_k,i_k}(x_k) \ \varphi\left(\frac{i_1+\ldots+i_k}m\right)\leq
\sum_{i=1}^k\frac{n_i}m\;\left(B_{m}\varphi\right)\left(x_i\right)	
\end{equation}
\par\medskip
for all convex functions $\varphi\in\C([0,1])$ and $x_1, \ldots, x_k \in [0,1]$, .
\end{theorem}

\begin{remark}
We note that, if $\varphi\in C([0,1])$ is a convex function, and
$$
g(x_1,\ldots,x_k)=\varphi\left( \frac{n_1}{m}x_1+ \ldots+ \frac{n_k}{m}x_k\right) \quad \text{for}\  (x_1,\ldots,x_k)\in D,
$$
then the inequality \eqref{eq:Gavv6} coincides with \eqref{eq:6prim}, which was proved in \cite{KomRaj2018bis}.
\end{remark}
\begin{remark}
If the inequality \eqref{eq:Gavv6} is satisfied for all $(x_1,\ldots,x_k)\in D$, then  
\begin{multline}\label{rm:gav13}
\left(B_{n_1,\ldots,n_k}g\right)(x_1,\ldots,x_k)+\left(B_{n_1,\ldots,n_k}g\right)(x_2,\ldots,x_k,x_1)+\ldots +\left(B_{n_1,\ldots,n_k}g\right)(x_k,x_1,\ldots,x_{k-1}) \\ 
\leq 
\sum_{i=1}^k\;\left(B_{n_1,\ldots,n_k}g\right)\left(x_i,\ldots,x_i\right)
\end{multline}
for all $(x_1,\ldots,x_k)\in D.$
\end{remark}

\textbf{Problem 3'.} With assumptions such as in Problem 3., give a characterization of the class of convex functions $g$ defined on $D$ and satisfying \eqref{rm:gav13} for all $(x_1,\ldots,x_k)\in D$.
\par\medskip

B. Gavrea \cite{Gav2018} presented also the following open problem.
\par\medskip
\textbf{Problem 4. (\cite{Gav2018}, Problem 2.)}
If $a_{n,k}(x)=\binom{n+k}{k}\:(1-x)^{n+1}x^k$ and if $\varphi$ is a convex continuous function on $[0,1]$, prove or disprove the following inequality:
\begin{equation}\label{eq:Gavtres}
 \sum_{i=0}^{\infty} \sum_{j=0}^{\infty} \left(a_{n,i}(x) \:a_{n,j}(x)+a_{n,i}(y) \:a_{n,j}(y)-2a_{n,i}(x) \:a_{n,j}(y)\right)\varphi\left(\frac{i+j }{2n+i+j}\right)\geq0. 
\end{equation}
\par\medskip

We will show that \eqref{eq:Gavtres} is not valid (in general).
Let $\varphi(u)=u$ and let $x\neq y$.
Let $\mu$ be the negative binomial probability distribution with parameters $n+1$ and $x$.
Then $\mu$ is concentrated on the set of non-negative integers and it satisfies
$\mu(\{k\})=\binom{n+k}{k}\:(1-x)^{n+1}x^k=a_{n,k}(x)$ for $k=0,1,\dots$.
Similarly, let $\nu$ be the negative binomial probability distribution with parameters $n+1$ and $y$.
Let $F$ and $G$ be the cumulative distribution functions of $\mu$ and $\nu$, respectively.
If $x<y$, then $F(u)<G(u)$ for $u\geq0$. If $x>y$, then $F(u)>G(u)$ for $u\geq0$ (see proof of Lemma~2.5.c in~\cite{KomRaj2018}).
In both cases we have $F(u)=G(u)=0$ for $u<0$, thus $(F-G)*(F-G)(u)>0$ for $u>0$.
Observe that for every $u\geq0$ we have
$$\tfrac u{2n+u}=\tfrac u{2n}-\int_0^\infty\tfrac{4n}{(2n+y)^3}(u-y)_+\lambda(dy).$$
Consequently,
\begin{multline*}
\sum_{i=0}^\infty\sum_{j=0}^\infty\left(a_{n,i}(x) \:a_{n,j}(x)+a_{n,i}(y) \:a_{n,j}(y)-2a_{n,i}(x) \:a_{n,j}(y)\right)\tfrac{i+j}{2n+i+j}=
\int_{-\infty}^\infty\tfrac u{2n+u}(\mu*\mu+\nu*\nu-2\mu*\nu)(du)=\\
\int_{-\infty}^\infty\tfrac u{2n}(\mu*\mu+\nu*\nu-2\mu*\nu)(du)
-\int_{-\infty}^\infty\int_0^\infty\tfrac{4n}{(2n+y)^3}(u-y)_+\lambda(dy)(\mu*\mu+\nu*\nu-2\mu*\nu)(du)=\\
0-\int_0^\infty\tfrac{4n}{(2n+y)^3}\int_{-\infty}^\infty(u-y)_+(\mu*\mu+\nu*\nu-2\mu*\nu)(du)\lambda(dy)=\\
-\int_0^\infty\tfrac{4n}{(2n+y)^3}((F-G)*(F-G))(y)\lambda(dy)<0.
\end{multline*}
It follows that the inequality \eqref{eq:Gavtres} is not valid for $\varphi(u)=u$.
However, if $u\mapsto \varphi\left(\frac{u}{n+u}\right)$ is convex on $[0,\infty)$
(e.g. if $\varphi$ is convex and decreasing),
then inequality \eqref{eq:Gavtres} is valid (cf. Theorem~\ref{th:necsuf} above and Theorem~2.6.c in~\cite{KomRaj2018}).
Using the same method it can be shown that if $u\mapsto \varphi\left(\frac{u}{n+u}\right)$ is concave on $[0,\infty)$ but it is not linear
(e.g. if $\varphi$ is concave and strictly increasing), then \eqref{eq:Gavtres} is not valid.






\bibliographystyle{elsarticle-num}



\end{document}